\documentclass[11pt,leqno]{amsart}
\usepackage{amsmath,amssymb,amsthm}
\usepackage{hyperref}
\usepackage{color}

\usepackage{bbm}

\newcommand{\R}{\mathbb{R}}

\DeclareMathOperator{\co}{co}

\renewcommand{\geq}{\geqslant}
\renewcommand{\leq}{\leqslant}

\newtheorem{theorem}{Theorem}[section]
\newtheorem{lemma}[theorem]{Lemma}

\newtheorem{proposition}[theorem]{Proposition}

\theoremstyle{definition}
\newtheorem{definition}[theorem]{Definition}

\theoremstyle{remark}

\numberwithin{equation}{section}

\def\fnote#1{\footnote}

\def\natu{{\mathbb N}}

\def\ignora#1{}
\def\n3#1{\left\vert  \! \left\vert \! \left\vert \, #1 \, \right\vert \!
  \right\vert \! \right\vert }

\begin{document}

\title{Banach spaces containing $c_0$ and elements in the fourth dual}

\author[Avil\'es]{Antonio Avil\'es}
\address[Avil\'es]{Universidad de Murcia, Departamento de Matem\'{a}ticas, Campus de Espinardo 30100 Murcia, Spain
	\newline
	\href{https://orcid.org/0000-0003-0291-3113}{ORCID: \texttt{0000-0003-0291-3113} } }
\email{\texttt{avileslo@um.es}}

\author[Mart\'inez-Cervantes]{Gonzalo Mart\'inez-Cervantes}
\address[Mart\'inez-Cervantes]{Universidad de Alicante, Departamento de Matem\'{a}ticas, Facultad de Ciencias, 03080 Alicante, Spain
	\newline
	\href{http://orcid.org/0000-0002-5927-5215}{ORCID: \texttt{0000-0002-5927-5215} } }	

\email{gonzalo.martinez@ua.es}

\author[Rueda Zoca]{Abraham Rueda Zoca}
\address[Rueda Zoca]{Universidad de Murcia, Departamento de Matem\'{a}ticas, Campus de Espinardo 30100 Murcia, Spain
	\newline
	\href{https://orcid.org/0000-0003-0718-1353}{ORCID: \texttt{0000-0003-0718-1353} }}
\email{\texttt{abraham.rueda@um.es}}
\urladdr{\url{https://arzenglish.wordpress.com}}

\thanks{The three authors have been supported by MTM2017-86182-P (Government of Spain, AEI/FEDER, EU) and Fundaci\'on S\'eneca, ACyT Regi\'on de Murcia grant 20797/PI/18. G. Mart\'{\i}nez-Cervantes has been co-financed by the European Social Fund (ESF) and the Youth European Initiative (YEI) under the Spanish S\'eneca Foundation (CARM) (ref. 21319/PDGI/19). A. Rueda Zoca was supported by Juan de la Cierva-Formaci\'on fellowship FJC2019-039973, by MICINN (Spain) Grant PGC2018-093794-B-I00 (MCIU, AEI, FEDER, UE), and by Junta de Andaluc\'ia Grants A-FQM-484-UGR18 and FQM-0185.}



\begin{abstract} 
A recent result of T.~Abrahamsen, P.~H\'ajek and S.~Troyanski states that a separable Banach space is almost square if and only if there exists $h\in S_{X^{****}}$ such that $\|x+h\|=\max\{\|x\|,1\}$ for all $x\in X$. The proof passes through a sequential version of being almost square which we call being \textit{sequentially almost square}. In this article we study these conditions in the nonseparable setting. On one hand, we show that a Banach space $X$ contains a copy of $c_0$ if and only if there exists an equivalent renorming  $| \cdot |$ on $X$ for which there exists $h\in S_{X^{****}}$ such that $|x+h|=\max\{|x|,1\}$ for every $x\in X$.
On the other hand, although it is unclear whether the aforementioned result of T.~Abrahamsen et al. holds in the nonseparable setting, 
we show that, under the existence of selective ultrafilters, if $X$ is a sequentially almost square Banach space then there exists $h\in S_{X^{****}}$ such that $\|x+h\|=\max\{\|x\|,1\}$ for all $x\in X$.
\end{abstract}

\maketitle

\section{Introduction}

The study of sequences equivalent to the $\ell_1$ or the $c_0$ basis has been a central topic in Banach space theory. In his celebrated paper \cite{maurey}, B.~Maurey characterised those Banach spaces $X$ containing $\ell_1$ or $c_0$ in terms of the presence of particular elements in succesive duals of $X$. Namely, he proved that a separable Banach space $X$ contains an isomorphic copy of $\ell_1$ if, and only if, there exists a non-zero element $x^{**}\in X^{**}$ so that
$$\Vert x+x^{**}\Vert=\Vert x-x^{**}\Vert$$
holds for every $x\in X$. In the same paper, he proved that a separable Banach space $X$ contains $c_0$ if, and only if, there are non-zero elements $x_1\in X^{**}, x_2\in X^{(4)}, x_3\in X^{(6)}$ and $x_4\in X^{(8)}$ so that
$$\Vert x+x_1-x_2\Vert=\Vert x-x_1+x_2\Vert=\Vert x+x_1-x_2+x_3-x_4\Vert$$
holds for every $x\in X$, where by $X^{(n)}$ we denote the $n$th-dual space of $X$. By the end of the eighties, the previous characterisation for spaces containing $\ell_1$ was put further by G. Godefroy in \cite[Theorem II.4]{god} in the following sense

\begin{theorem}\label{theo:godefroy}
 Given a Banach space $X$, the following are equivalent:
\begin{enumerate}
\item $X$ contains an isomorphic copy of $\ell_1$.
\item\label{condi-Lorto} There exists an equivalent renorming $\vert\cdot\vert$ on $X$ with a non-zero element $x^{**}\in X^{**}$ so that
$$\vert x+x^{**}\vert=\vert x\vert+\vert x^{**}\vert$$
holds for every $x\in X$.
\item\label{condi-octah} There exists an equivalent renorming $\vert\cdot\vert$ on $X$ so that, for every finite-dimensional subspace $F$ of $X$ and every $\varepsilon>0$, there exists $x\in S_X$ so that
$$\vert y+\lambda x\vert\geq (1-\varepsilon)(\vert y\vert+\vert\lambda\vert)$$
holds for every $y\in F$ and every $\lambda\in\mathbb R$.
\end{enumerate}
\end{theorem}

Note that condition (\ref{condi-Lorto}) is weaker than the original Maurey condition in the sense that it requires an equivalent renorming but, on the other hand, the element of the bidual in Godefroy theorem has a rather stronger $\ell_1$ flavour.

New properties have been introduced and studied where the role of $\ell_1$ and norms obtained by addition is replaced by $c_0$ and norms obtained by taking a maximum. In this setting, the analogue of (\ref{condi-octah}) was defined by Abrahamsen, Langemets and Lima \cite{all}: a Banach space $X$ is said to be \textit{almost square (ASQ)} if for every $\varepsilon>0$ and every $x_1, \ldots, x_n \in S_X$ there exists $h\in S_X$ such that $\|x_i \pm h\| \leq 1+\varepsilon$ for every $i=1,\ldots,n$. Becerra, López-Pérez and Rueda Zoca proved the following:

\begin{proposition}\label{Propc0implicaASQ}	{\cite[Theorem 2.3]{blr}}
	A Banach space $X$ contains a copy of $c_0$ if and only if $X$ admits an equivalent norm that is ASQ.
\end{proposition}

A $c_0$-analogous condition for (\ref{condi-Lorto}) was given by Farmaki \cite{Far}: There exists $h\in S_{X^{****}}$ so that, for every $x\in X$, the following equation holds
\begin{equation}\label{equation:ortofarmaki}\Vert x+h\Vert=\max\{\Vert x\Vert, 1\}.
\end{equation}
One needs to get to the fourth dual in this case. If one thinks of $X=c_0$, it is impossible to find such $h$ in $\ell_\infty$ but a limit of the canonical basis through a free ultrafilter in the weak$^*$ topology of $\ell_\infty^{**}$ does the job. It is easy to see that Farmaki's condition implies ASQ, while T.~Abrahamsen, P.~H\'ajek and S.~Troyanski showed that the converse holds in separable Banach spaces.

\begin{proposition}\label{PropEquivalencesSeparable}	{\cite[Proposition 4.2]{AHT}}
	For a separable Banach space $X$, the following are equivalent:
	\begin{enumerate}
		\item $X$ is ASQ;
		\item There exists $h\in S_{X^{****}}$ such that $\|x+h\|=\max\{\|x\|,1\}$ for all $x\in X$.
	\end{enumerate}
\end{proposition}

One of our main results in this article asserts that a renorming result for this property holds in the nonseparable setting:

\begin{theorem}
	\label{theo:renorming}
	Let $X$ be a Banach space containing an isomorphic copy of $c_0$. Then there exists an equivalent renorming $| \cdot |$ on $X$ for which there exists $h\in S_{X^{****}}$ such that $|x+h|=\max\{|x|,1\}$ for every $x\in X$.
\end{theorem}

The proof of Proposition \ref{PropEquivalencesSeparable} passed through another property: We say that a Banach space $X$ is \textit{sequentially ASQ} if there exists a sequence $(x_n)$ in $S_X$ so that
$$
\Vert x+x_n\Vert\longrightarrow \max\{\|x\|,1\}
$$
for every $x\in X$. It is immediate that if $X$ is sequentially ASQ then it is ASQ, while the converse holds for separable Banach spaces \cite[Lemma 4.4]{AHT} (see also the proof of \cite[Corollary 2.4]{gr}). In \cite{AHT} this separable equivalence is used in combination with the work of Farmaki \cite[Proposition 2.10]{Far}. Farmaki's techniques strongly depend on the separability assumption because they rely on a diagonalization argument. Consequently, it is unclear whether ASQ, or even sequential ASQ, implies the existence of a non-zero element satisfying \eqref{equation:ortofarmaki} holds in the nonseparable setting. The other main result of this article is a consistent answer to this question:

\begin{theorem}\label{theo:selectivo}
	Let $X$ be a sequentially ASQ Banach space and let $\mathcal U$ be a selective ultrafilter over $\mathbb N$. Then there exists $h\in S_{X^{****}}$ such that
	$$\Vert x+h\Vert=\max\{\Vert x\Vert,1\}$$
	holds for all $x\in X$.
\end{theorem}

Let us recall that the existence of selective ultrafilters over $\mathbb{N}$ is independent of ZFC. In particular, they do exist under CH. Intuitively, the role that selective ultrafilters play in the previous theorem is that they allow us to use Farmaki's ideas avoiding the above mentioned diagonalization argument. A similar use of selective ultrafilters is made in \cite[Section 5]{amcrz} with $\ell_1$-sequences. We do not know if Theorem \ref{theo:selectivo} holds without the use of additional axioms, neither if every Banach space containing $c_0$ can be renormed to be sequentially ASQ.


\textbf{Notation.} Every Banach space in this article is assumed to be a real Banach space. A family of sets is said to have the \textit{finite intersection property} if any intersection of finitely many sets of the family is nonempty.

A nonprincipal ultrafilter $\mathcal{U}$ is \emph{selective} (a.k.a. \emph{Ramsey}) if for every partition $\mathbb{N} = \bigcup_n A_n$ into sets $A_n\not\in\mathcal{U}$ there exists a set in the ultrafilter that meets each set of the partition at exactly one point. Selective ultrafilters exist under CH, even under $\mathfrak{p}=\mathfrak{c}$, and in many other models of set theory, cf. \cite[Proposition 10.9, Chapter 25]{halbeisen} and \cite[Section 4.5]{barto}. We will not work with the previous definition but with the following characterization due to Mathias \cite[Theorem 2.12]{mathias}.

\begin{theorem}\label{Mathiastheorem}
	A nonprincipal ultrafilter $\mathcal U$ is selective if and only if the intersection of $\mathcal{U}$ with any dense analytic ideal is nonempty.
\end{theorem}

Let us explain the terminology used. An ideal $I$ is called \textit{dense} (or \emph{tall}) if for every infinite set $A\subset\mathbb{N}$ there exists a further infinite subset $B\subset A$ with $B\in I$.. A \emph{Polish space} is a topological space that is homeomorphic to a separable complete metric space. We can view $\mathcal{P}(\mathbb{N})$ as a Polish space through the natural bijection with $\{0,1\}^\mathbb{N}$ that identifies each set $A$ with its characteristic function. A subset $Y$ of a Polish space is called \emph{analytic} if there is a continuous surjection from a Polish space $C$ onto $Y$. We refer to \cite{Kechris} for the basics. 

%

\section{Further examples and preliminaries}

As we have pointed out in the introduction, every ASQ Banach space contains $\varepsilon$-copies of $c_0$. 
Infinite-dimensional $C(K)$-spaces are likely the most natural examples of Banach spaces containing isometric copies of $c_0$. Nevertheless, it is easy to check that if $x_1$ is the constant function one in a $C(K)$-space then no element $h$ satisfying the condition of the definition of an ASQ space exists. In order to avoid this problem one may think of $C_0(L)$, i.e. the space of continuous functions defined on a locally compact Hausdorff space $L$ which vanish at infinity endowed with the supremum norm. 

\begin{proposition}\label{prop:caraC_(L)asq}
Let $L$ be a locally compact Hausdorff space. Then $C_0(L)$ is ASQ if and only if $L$ is not compact.
\end{proposition}

\begin{proof}
The paragrapgh preceding the statement of the proposition yields that if $C_0(L)$ is ASQ then $L$ is not compact. Conversely, assume that $L$ is not compact and let us prove that $X=C_0(L)$ is ASQ. Remember that $C_0(L)$ can be viewed as the space of continuous functions on the one-point compactification $L\cup\{\infty\}$ that vanish at $\infty$. Pick $f_1,\ldots, f_n\in S_X$ and $\varepsilon>0$, and let us find $g\in S_X$ so that $\Vert f_i+g\Vert\leq 1+\varepsilon$ holds for every $1\leq i\leq n$. Consider, for every $i\in\{1,\ldots, n\}$, the set
$$K_i:=\{t\in L: \vert f_i(t)\vert\geq \varepsilon\},$$
which is compact by the definition of $C_0(L)$. Consequently, $K:=\bigcup\limits_{i=1}^n K_i$ is compact too. Since $L$ is not compact, we can find $t_0\in L\setminus K$ and, by Urysohn's Lemma, a continuous function $g:L\longrightarrow \mathbb [0,1]$ with $g(t_0)=1$ while $g$ vanishes at $K\cup\{\infty\}$. From this we get that  $\Vert f_i+g\Vert\leq 1+\varepsilon$ as desired.
\end{proof}

Characterizing those $C_0(L)$-spaces which are sequentially ASQ requires a little more work. First, let us introduce some notation:
\begin{definition}
Let $L$ be a locally compact Hausdorff space. We say that a sequence of nonempty open sets $(U_n)_n$ in $L$ is \textit{compact escaping} if for every compact set $K\subseteq L$ there exists $m \in \natu$ such that $U_n \cap K = \emptyset $ for every $n \geq m$.
\end{definition}

\begin{proposition}\label{prop:caraC_0(L)seqasq}
Let $L$ be a locally compact Hausdorff space. Then $C_0(L)$ is sequentially ASQ if and only if $L$ has a  compact escaping sequence.
\end{proposition}
\begin{proof}
Suppose first that $C_0(L)$ is sequentially ASQ (witnessed by $(f_n)$). Set $U_n=\{t\in L: |f_n(t)|>\frac{1}{2}\}$. We claim that the sequence of open sets $(U_n)_n$ is compact escaping. Let $K$ be any compact set in $L$. Then, we can take a function $f \in S_{C_0(L)}$ taking the value one on $K$. Since 
$$ \| f \pm f_n \| \longrightarrow 1 ,$$
there exists $m\in \natu$ such that $\| f \pm f_n \| <\frac{3}{2}$ for every $n\geq m$. But then $U_n \cap K = \emptyset$ for every $n \geq m $ as desired.

In order to prove the reverse implication, we fix a compact escaping sequence $(U_n)_n$ in $L$. For each $n$, we can find $f_n \in S_{C_0(L)}$ that vanishes out of $U_n$. The fact that $(\|f+f_n\|)_n$ converges to $\max\{\|f\|,1\}$ for every $f\in C_0(L)$ follows from the definition of a compact escaping sequence and the fact that $\{t\in L: |f(t)|\geq \varepsilon \}$ is compact for every $\varepsilon >0$ and every $f \in C_0(L)$.
\end{proof}

If $K$ is a compact space, and $x$ is a non-isolated  point of $K$ that is not the limit of a sequence in $K\setminus\{x\}$, say for instance $\omega_1$ in the ordinal interval $[0,\omega_1]$, then $K\setminus\{x\}$ is a locally compact space without compact scaping sequences. This provides examples of ASQ nonsequentially ASQ spaces. Notice that the existence of a sequence in a locally compact Hausdorff space $L$ converging to $\infty$ does not imply the existence of a compact escaping sequence; for instance, this is the case if we remove $(\omega,\omega_1)$ from $[0,\omega]\times[0,\omega_1]$. More examples of ASQ spaces (respectively sequentially ASQ spaces) can be found in \cite{AHT,all,blr} (\cite{amcrz,gr}).

We finish this section with some auxiliary lemmas. The first one provides sufficient conditions for the existence of an element $h \in S_{X^{****}}$ satisfying condition (3) in Proposition \ref{PropEquivalencesSeparable}.

\begin{lemma}
\label{lem:orthogonalfourthdual}
Let $X$ be a Banach space and $\{C_\alpha: \alpha \in \Gamma\}$ a family of bounded convex sets in $X^{**}$ with the finite intersection property such that for every $\varepsilon>0$ and every $x\in X$ there exists $\alpha \in \Gamma$ satisfying $$(1-\varepsilon)\max\{\|x\|,1\}\leq \|x+x^{**}\| \leq (1+\varepsilon)\max\{\|x\|,1\}$$
for every $x^{**} \in C_\alpha$.
Then there exists $h\in S_{X^{****}}$ such that $\|x+h\|=\max\{\|x\|,1\}$ for every $x\in X$.
\end{lemma}
\begin{proof}
Since the family $\{C_\alpha: \alpha \in \Gamma\}$ has the finite intersection property, we have that
$ \bigcap_{\alpha\in \Gamma} \overline{C_\alpha}^{w^*} \neq \emptyset $, where the previous weak*-closures are taken in $X^{****}$. We claim that any $h \in \bigcap_{\alpha\in \Gamma} \overline{C_\alpha}^{w^*}$ satisfies the desired property.
Fix $x \in X$ and an arbitray $\varepsilon >0$. Then there exists $\alpha \in \Gamma$ such that $$(1-\varepsilon)\max\{\|x\|,1\}\leq \|x+x^{**}\| \leq (1+\varepsilon)\max\{\|x\|,1\|\}$$
for every $x^{**} \in C_\alpha$.
Since $h \in \overline{C_\alpha}^{w^*}$, a standard application of Hahn-Banach Theorem and the $w^*$-lower semicontinuity of the norm of $X^{****}$ ensure that 
$$(1-\varepsilon)\max\{\|x\|,1\}\leq \|x+h\| \leq (1+\varepsilon)\max\{\|x\|,1\|\}.$$
The conclusion follows from the arbitrariness of $\varepsilon>0$.
\end{proof}


\begin{lemma}\label{lemma:auxiliar1}
	Let $X$ be a sequentially ASQ  Banach space, witnessed by $(x_n)$. For every finite-dimensional subspace $F$ of $X$ and every $\varepsilon>0$ there exists $m\in \natu$ such that
	$$\left\vert \Vert y+\lambda x_n\Vert-\max\{\Vert y\Vert, |\lambda| \}\right\vert<\varepsilon \max\{\Vert y\Vert, |\lambda| \}$$
	for every $y\in F$, every $\lambda \in \R$ and every $n \geq m$.
\end{lemma}

\begin{proof} 
	Fix $F$ and $\varepsilon>0$. Since $S_F$ is compact, we can find $\{y_1,\ldots,y_k\}\subseteq S_F$ such that $S_F\subseteq \bigcup\limits_{i=1}^k B\left( y_i,\frac{\varepsilon}{2}\right)$. By definition, we can find $m\in\mathbb N$ such that $n\geq m$ implies
	$$|\Vert y_i \pm x_n\Vert - 1|< \frac{\varepsilon}{2}\ \mbox{ for every } i\in\{1,\ldots, k\}.$$
	Fix $n\geq m$. Since $S_F\subseteq \bigcup\limits_{i=1}^n B\left( y_i,\frac{\varepsilon}{2}\right)$, the triangle inequality implies that
	$$|\Vert y \pm x_n\Vert -1|< \varepsilon$$
	holds for all $y\in S_F$.
	Now the conclusion follows from \cite[Lemma 2.2]{all}.
\end{proof}

%

\begin{lemma}
	\label{LemmaASQSubsequence}
	Let $X$ be a sequentially ASQ Banach space, witnessed by $(x_n)$, and let $\varepsilon_n$ be a sequence of positive numbers such that $\varepsilon_n\rightarrow 0$. Let $F_n$ be an incresing sequence of finite-dimensional Banach spaces. Then, there exists a subsequence $(y_n)$ of  $(x_n)$ such that
	$$\left\vert \Vert x+y\Vert-\max\{\Vert x\Vert, 1\}\right\vert<\varepsilon_n \max\{\Vert x\Vert, 1\}$$
	holds for every $n\in\mathbb N, x\in span(F_n\cup\{y_1,\ldots, y_{n-1}\})$ and  every 
	$y\in\{\sum_{i=n+1}^k\lambda_i y_i: k>n,~ \max\{\vert \lambda_i\vert: n+1\leq i \leq k\}= 1\}$.
\end{lemma}

\begin{proof}
	Take a sequence of positive numbers $\delta_n>0$ such that $$1-\varepsilon_n<\prod_{i=n+1}^\infty (1-\delta_i)<\prod_{i=n+1}^\infty (1+\delta_i) <1+\varepsilon_n$$ for every $n\in \natu$.	
	
	An inductive application of Lemma \ref{lemma:auxiliar1} yields a subsequence $(y_n)$ of $(x_n)$ such that
	$$\left\vert \Vert x+\lambda y_n\Vert-\max\{\Vert x\Vert, |\lambda| \}\right\vert<\delta_n \max\{\Vert x\Vert, |\lambda| \}$$
	for every $x\in span(F_n\cup\{y_1,\ldots, y_{n-1}\})$, every $\lambda \in \R$ and every $n \in \natu$.

	We prove now that the subsequence $(y_n)$ satisfies the thesis of the lemma.
	Set $n\in\mathbb N,~x\in span(F_n\cup\{y_1,\ldots, y_{n-1}\})$ and $y=\sum_{i=n+1}^k\lambda_i y_i$ with $ \max\{\vert \lambda_i\vert: n+1\leq i\leq k\}= 1\}$.
	
	Then,
	\[\begin{split} 	\Vert x+y\Vert&  = \left\Vert x+\sum_{i=n+1}^k\lambda_i y_i\right\Vert\\
	&  =  \left\Vert \left(x+\sum_{i=n+1}^{k-1}\lambda_i y_i \right) +\lambda_ky_k \right\Vert \\
	& < (1+\delta_k)\left(\max\left\{\left\|x+\sum_{i=n+1}^{k-1}\lambda_i y_i\right\|, |\lambda_k|\right\}\right) \\
	& < (1+\delta_k)(1+\delta_{k-1})\left(\max\left\{\left\|x+\sum_{i=n+1}^{k-2}\lambda_i y_i\right\|, |\lambda_{k-1}|,|\lambda_k|\right\}\right)\\
	& <\ldots < \left(\prod_{i=n+1}^k(1+\delta_i) \right)\max\{\|x\|,|\lambda_{n+1}|,\lambda_{n+2}\|,\ldots,|\lambda_k|\}\\
	& <(1+\varepsilon_n)\max\{\|x\|,1\}.\end{split}\]

	The proof of the inequality $$\|x+y\| > (1-\varepsilon) \max\{\|x\|,1\}$$
	is analogous.
\end{proof}

\section{Sequentially ASQ spaces and selective ultrafilters}

The goal of this section is to prove Theorem \ref{theo:selectivo}. Throughout this section we fix a sequentially ASQ Banach space $X$, witnessed by a sequence of vectors $(x_n)$, and a selective ultrafilter $\mathcal{U}$.

%
%

If we fix a sequence $(\varepsilon_n)$ converging to 0, we can assume, up to an application of Lemma \ref{LemmaASQSubsequence}, that $(x_n)$ satisfies that the inequalities
\begin{equation}
\label{eqbounds}
1-\varepsilon_n\leq \Vert x_i\pm y\Vert\leq 1+\varepsilon_n
\end{equation}
hold for every $1\leq i\leq n$ and every $y\in\{\sum_{j=n+1}^k\lambda_j x_j: k\in\mathbb N, \max\{\vert \lambda_j\vert, n+1\leq j\leq k\}= 1\}$.

%
%

Given $A\in\mathcal U$ and $n\in \natu$ we define
$$x_{n}^A:=\sum_{i\in A,~i\leq n}x_i.$$
It follows from (\ref{eqbounds}) that the set $\{x_{n}^A: n \in \natu, A\in\mathcal U\}$ is bounded. For each $A \in \mathcal{U}$ we fix $y_A^{**}$ any $w^*$-cluster point of the sequence $(x_{n}^A)_{n}$ in $B_{X^{**}}$ and consider $C_A:=\co \{y_B^{**}: B\subseteq A,~B\in \mathcal U\} \subseteq B_{X^{**}}$.
In order to prove Theorem \ref{theo:selectivo}, it is enough to show that the family $\{C_A: A \in \mathcal{U}\}$ satisfies the hypotheses of Lemma \ref{lem:orthogonalfourthdual}. The finite intersection property is a consequence of the finite intersection property of the ultrafilter $\mathcal{U}$, so what remains to be proved is the following lemma:

\begin{lemma}
\label{LemmaConvexBounds} Let $x\in X$ and $\varepsilon>0$. Then there exists $A \in \mathcal{U}$ such that 
$$(1-\varepsilon)\max\{\|x\|,1\} \leq \|x+x^{**}\| \leq (1+\varepsilon)\max\{\|x\|,1\}$$
for every $x^{**}\in C_A$.
\end{lemma}
\begin{proof}
	Define $J$ as the family of those $A\subseteq \mathbb N$ such that
	$$(1-\varepsilon)\max\{\Vert x\Vert,1\}\leq \Vert x+y\Vert\leq (1+\varepsilon)\max\{\Vert x\Vert, 1\}$$
	holds for every $y\in\lbrace\sum\limits_{\substack{i\in B}}\lambda_i x_{i}: B\subseteq A,~B \mbox{ finite and } \max\{\vert \lambda_i\vert: i \in B\}= 1\rbrace$.

	We claim that the family $J$ is a closed subset of $\mathcal{P} (\natu)$. Suppose that $A\notin J$. Then there exist $B\subseteq A$ with $B$ being finite and a family of real numbers $\{\lambda_i: i \in B\}$ with $\max\{\vert \lambda_i\vert: i \in B\}= 1$ such that 
	$$ |\Vert x+y\Vert - \max\{\Vert x\Vert, 1\} | >\varepsilon,$$
	where $y=\sum\limits_{i\in B}\lambda_i x_{i}$. It is immediate then that the set $U$ of all subsets of $\mathcal{P}(\natu)$ containing $B$ is an open subset in $\mathcal{P} (\natu)$ with $A\in U$ and $U\cap J = \emptyset$, so we can conclude that $J$ is closed.

	Let $I$ be the ideal generated by $J$. On one hand, since $J$ is closed, we have that $I$ is an analytic set (cf. \cite[Fact 5.2]{amcrz}). On the other hand, Lemma \ref{LemmaASQSubsequence} guarantees that $I$ is a dense ideal in $\mathbb N$. By Theorem \ref{Mathiastheorem} there exists a set $A\in J\cap \mathcal U$. We claim that $C_A$ satisfies the desired property. Let $x^{**}\in C_A$. Then $x^{**}$ is of the form $\sum_{j=1}^n \lambda_j y_{B_j}^{**}$, where $\lambda_j \geq 0$, $\sum_{j=1}^n \lambda_j=1$,  $B_j \subseteq A$ and $B_j \in \mathcal{U}$ for every $j$.
	Set $m \in B_1 \cap B_2 \cap \ldots \cap B_n$. Notice that the previous intersection is nonempty since it is a set of $\mathcal{U}$.
	Consider $C_{A,m}':=\co\{x_n^B: n\in \natu,~m\in B\subseteq A,~B\in \mathcal{U},~n\geq m\}\subseteq X$. It is clear that $x^{**}\in \overline{C_{A,m}'}^{w^*}$ in $X^{**}$.
	Furthermore, any element of $C_{A,m}'$ is of the form $\sum\limits_{\substack{i\in B}}\lambda_i x_{i}$ with $m\in B\subseteq A,~B \mbox{ finite and } \max\{\vert \lambda_i\vert: i \in B\}=\lambda_m= 1$.
	Thus, since $A \in J$, we conclude from the definition of $J$ that 
	$$(1-\varepsilon)\max\{\Vert x\Vert,1\}\leq \Vert x+y\Vert\leq (1+\varepsilon)\max\{\Vert x\Vert, 1\}$$
	for every $y\in C_{A,m}'$.
	It follows from a straightforward application of Hahn-Banach Theorem and the $w^*$-lower semicontinuity of the norm of $X^{**}$ that 
	$$(1-\varepsilon)\max\{\Vert x\Vert,1\}\leq \Vert x+y^{**}\Vert\leq (1+\varepsilon)\max\{\Vert x\Vert, 1\}$$
	for every $y^{**}\in \overline{C_{A,m}'}^{w^*}$ and, in particular, $$(1-\varepsilon)\max\{\Vert x\Vert,1\}\leq \Vert x+x^{**}\Vert\leq (1+\varepsilon)\max\{\Vert x\Vert, 1\}$$
	as desired.
	\end{proof}


\section{Renorming spaces containing $c_0$}

%

The aim of this section is to prove Theorem \ref{theo:renorming}.
Set $X$ any Banach space containing an isomorphic copy of $c_0$. Throughout this section $\mathcal{U}$ denotes a fixed free ultrafilter.
The desired equivalent norm will be the one defined in \cite[Theorem 2.3]{blr}. Let us recall the definition of this norm.

We can assume, up to an equivalent renorming, that we have a sequence $(e_n)$ in $X$ isometric to the $c_0$-basis. The weak$^\ast$-closure of $c_0$ will produce a copy of $\ell_\infty$ inside $X^{**}$ and we will have $X^{**}=\ell_\infty \oplus Z$ (c.f. e.g. \cite[Proposition 5.13]{fab}).  Notice that if we write $e_n=y+z$ with $y\in \ell_\infty$ and $z\in Z$ then $y=e_n$ and $z=0$. We consider on $X$ an equivalent norm $\vert\cdot \vert$ as follows. If we see $X\subseteq X^{**}=\ell_\infty\oplus Z$, given any $x\in X$ there are unique $y\in \ell_\infty$ and $z\in Z$ so that $x=(y,z)$.
We declare the following norm:
	$$\vert x\vert:=\max\{\vert \lim_\mathcal U y\vert, \sup_{n\in\mathbb N}\vert y(n)-\lim_\mathcal{U}y\vert,\Vert z\Vert\}.$$
	This defines an equivalent norm on $X$; see \cite[Theorem 2.3]{blr} for a proof. 
	Furthermore, it is immediate that $(e_n)$ is isometrically equivalent to the $c_0$-basis under this norm.

	For every $n\in\mathbb N$ and every $A\in \mathcal U$ we define
	$$x_n^A:=\sum_{i\in A,~i\leq n} e_i.$$
	Define, given any $A\in\mathcal U$, 
	$$y_A^{**}:=\lim_{n,\mathcal U} x_n^A\in X^{**},$$
	where the previous limit is taken in the weak* topology of $X^{**}$, and 
	set $C_A:=\co \{y_B^{**}: B\subseteq A,~B\in \mathcal U\}$.
	The proof now follows the same ideas of the previous section. It is enough to show that the family $\{C_A: A \in \mathcal{U}\}$ satisfies the hypotheses of Lemma \ref{lem:orthogonalfourthdual}. The finite intersection property is again a direct consequence of the finite intersection property of $\mathcal{U}$ and the rest follows from the following lemma:

%
%
%
%


\begin{lemma}
	\label{LemmaConvexBoundsRenorming} Let $x\in X$ and $\varepsilon>0$. Then there exists $A \in \mathcal{U}$ such that 
	$$(1-\varepsilon)\max\{|x|,1\} \leq |x+x^{**}| \leq (1+\varepsilon)\max\{|x|,1\}$$
	for every $x^{**}\in C_A$.
\end{lemma}	
\begin{proof}
		Given $x\in X$ and $\varepsilon>0$, define
		$$A:=\left\{m\in\mathbb N: \left\vert y(m)-\lim_\mathcal U y\right\vert<\varepsilon\right\},$$
		where $y$ is the unique $y\in\ell_\infty$ so that $x=(y,z)\in \ell_\infty\oplus Z$.
		Notice that $A\in \mathcal{U}$.
		We claim that $C_A$ has the desired property. Let $x^{**}=\sum_{i=1}^n \lambda_i y_{B_i}^{**} \in C_A$, where $\lambda_1,\ldots, \lambda_n\in [0,1]$, $\sum_{i=1}^n\lambda_i=1$ and $B_1,\ldots, B_n\in \mathcal U$ are contained in $A$.
		Pick $p\in \bigcap\limits_{i=1}^n B_i\in \mathcal U$. Note that $\sum_{i=1}^n \lambda_i x_p^{B_i}\in c_0$, so $\lim_\mathcal U \sum_{i=1}^n \lambda_i x_p^{B_i} =0$ and therefore
		$$\left\vert x\pm\sum_{i=1}^n \lambda_i x_p^{B_i}\right\vert=\max\left\{\left\vert \lim_\mathcal U y \right\vert,\sup_{q\in\mathbb N} \left\vert y(q)\pm\sum_{i=1}^n \lambda_i x_p^{B_i}(q)-\lim_\mathcal U y \right\vert  , \Vert z\Vert\right\}.$$
		
		Let us now estimate
		$\sup_{q\in\mathbb N} \vert y(q)\pm\sum_{i=1}^n \lambda_i x_p^{B_i}(q)-\lim_\mathcal U y\vert$. Given $q\in\mathbb N$, we distinguish two cases. If $q\notin A$ then 
		$$\left\vert y(q)\pm\sum_{i=1}^n \lambda_i x_p^{B_i}(q)-\lim_\mathcal Uy\right\vert=\left\vert y(q)-\lim_\mathcal U(y)\right\vert.$$
		On the other hand, if $q\in A$ then $\vert y(q)-\lim_\mathcal U y\vert<\varepsilon$, so
		\[
		\begin{split}
		\left\vert y(q)\pm\sum_{i=1}^n \lambda_i x_p^{B_i}(q)-\lim_\mathcal U y\right\vert& \leq \left\vert y(q)-\lim_\mathcal U (y)\right\vert+\left\vert \sum_{i=1}^n \lambda_i x_p^{B_i}(q)\right\vert\\
		& \leq \varepsilon+\sum_{i=1}^n \lambda_i=1+\varepsilon.
		\end{split}
		\]
		As a consequence we get that
		\[
		\begin{split}
		\left\vert x\pm\sum_{i=1}^n \lambda_i x_p^{B_i}\right\vert\leq \max\left\{\vert \lim_\mathcal U y\vert, \sup\limits_{q\notin A} \vert y(q)-\lim_\mathcal U y\vert,1+\varepsilon,\|z\|\right\}\leq \max\{\vert x\vert, 1+\varepsilon\}.
		\end{split}
		\]
		Since the norm of $(X,\vert\cdot\vert)^{**}$ is weak*-lower semicontinuous we derive that
		$$\left\vert x\pm\sum_{i=1}^n \lambda_i y_{B_i}^{**}\right\vert\leq \liminf_{p,\mathcal U} \left\vert x+\sum_{i=1}^n \lambda_i x_p^{B_i}\right\vert\leq \max\{\vert x\vert, 1+\varepsilon\}\leq (1+\varepsilon)\max\{|x|,1\}.$$
		Let us now prove the inequality from below. Take a functional $f_p\in (S_{X^*}, |\cdot|)$ with $f_p(e_p)=1$ and observe that, since $(e_n)$ is isometric to the $c_0$-basis, $f_p(e_k)=0$ holds for every $k\neq p$.
		Thus, if $r\geq p$ then $f_p(x_r^{B_i})=1$ and therefore $f_p(\sum_{i=1}^n \lambda_i x_r^{B_i})=1$.

		Since $y_{B_i}^{**}=\lim_{r,\mathcal U} x_r^{B_i}$ we conclude that $f_p(x^{**})=f_p(\sum_{i=1}^n \lambda_i y_{B_i}^{**})=1$ and therefore $|x^{**}|=1$. We have already shown in the first part of the proof that $\vert x\pm x^{**}\vert\leq (1+\varepsilon)\max\{\vert x\vert, 1\}$. We are going to prove now the other inequality $\vert x\pm x^{**}\vert\geq (1-\varepsilon) \max\{\vert x\vert, 1\}$. Suppose first that $\max\{\vert x\vert, 1\}=1$. We know that $f_p(x^{**})=1$. Now
		$$1+\varepsilon\geq \vert x\pm x^{**}\vert\geq f_p(x^{**})\pm f_p(x)\Longrightarrow 1+\varepsilon\geq 1+\vert f_p(x)\vert,$$
		which implies that $\vert f_p(x)\vert\leq \varepsilon$. Then
		$$\vert x\pm x^{**}\vert\geq f_p(x^{**}\pm x)\geq 1-\vert f_p(x)\vert\geq 1-\varepsilon.$$
		If $\max\{\vert x\vert, 1\}=\vert x\vert$, take $g\in (S_{X^*},|\cdot|)$ such that $g(x)=|x|$. Since $$(1+\varepsilon)|x| \geq \vert x\pm x^{**}\vert\geq g(x)\pm g(x^{**})=|x|\pm g(x^{**}),$$ we obtain that $|g(x^{**})|\leq \varepsilon|x|$ and
		$$\vert x\pm x^{**}\vert\geq g(x\pm x^{**})\geq |x|-\vert g(x^{**})\vert\geq |x|-\varepsilon|x| \geq (1-\varepsilon)|x|$$
		as desired.
	\end{proof}

\end{document}